\documentclass[11pt,
twoside,%
onecolumn,%
nofloats,%
nobibnotes,%
nofootinbib,%
superscriptaddress,%
noshowpacs,%
centertags]%
{revtex4}
\usepackage{ljm}
\usepackage{enumerate}

\newtheorem{propos}{Proposition}

\theoremstyle{definition}
\newtheorem{remark}{Remark} 

\begin{document}
\titlerunning{Balayage of Measures on the Complex Plane \dots}
\authorrunning{B. N. Khabibullin, E. B. Menshikova}

\title{Balayage of Measures on the Complex Plane with respect to Harmonic Polynomials and Logarithmic Kernels}

\author{\firstname{B. N.}~\surname{Khabibullin}}
\email[E-mail: ]{khabib-bulat@mail.ru } 
\affiliation{Bashkir State University, Bashkortostan, Russian Federation}

\author{\firstname{E. B.}~\surname{Menshikova}}
\email[E-mail: ]{algeom@bsu.bashedu.ru } 
\affiliation{Bashkir State University, Bashkortostan, Russian Federation}


\received{August 03, 2020}

\maketitle

\section{Balayage of measures}

As usual, $\mathbb N:=\{1,2, \dots\}$ is  set of all {\it natural\/} numbers. 
We denote singleton sets by a symbol without curly brackets. So,  ${\mathbb{N}}_0:=0\cup {\mathbb{N}}:=\{0\}\cup {\mathbb{N}}$. $\mathbb R$  is the \textit{real line,\/} or the {\it real axis\/} of the complex plane $\mathbb C$, with the standard {\it Euclidean norm-module\/} $|\cdot |$ and order $\leq$, ${\mathbb{R}}^+:=\{r\in {\mathbb{R}}\colon 0\leq r\}$ is the {\it positive closed semiaxis,\/} and ${\mathbb{R}}^+\!\setminus\!0$ is the {\it positive open semiaxis.\/}
The {\it extended real line\/}  $\overline {\mathbb{R}}:={\mathbb{R}}\cup \pm\infty$ is the order completion of ${\mathbb{R}}$ by the  union   with $+\infty:=\sup {\mathbb{R}}$ and $-\infty:=\inf {\mathbb{R}}$, and  $\inf \varnothing :=+\infty$, $\sup \varnothing :=-\infty$ for the {\it empty set\/} $\varnothing$ etc. Besides, $\overline {\mathbb{R}}^+:={\mathbb{R}}^+\cup+\infty$, 
$x\cdot (\pm\infty):=\pm\infty=:(-x)\cdot (\mp\infty)$ for $x\in \overline {\mathbb{R}}^+\!\setminus\!0$,  
but $0\cdot (\pm\infty):=0$ unless otherwise specified. 

Denote by ${\mathbb{C}}_{\infty}:={\mathbb{C}}\cup \{\infty\}$ the 
{\it  Alexandroff\/} one-point {\it compactification\/} of  the {\it complex plane\/} ${\mathbb{C}}$
with the standard {\it Euclidean norm-module\/} $|\cdot|$, but $|\infty|:=+\infty$. 
For $S\subset {\mathbb{C}}_{\infty}$  we let $\complement S :={\mathbb{C}}_{\infty}\!\setminus\!S$, ${\rm clos\,} S$, ${\rm int\,} S:=\complement ({\rm clos\,} \complement  S)$, and $\partial S:={\rm clos\,} S\!\setminus\!{\rm int\,} S$ denote its
{\it complement,\/} {\it closure,} {\it interior,} and {\it boundary\/}  always in ${\mathbb{C}}_{\infty}$, and $S$ is equipped with the topology induced from ${\mathbb{C}}_{\infty}$. If ${\rm clos\,} S'\subset S$, then we write   $S'\Subset S$.

 Let $X,Y$ are sets. We denote by $Y^X$ the set of all functions  $f\colon X\to Y$.
 
Let ${\text{\sf Bor}} (S)$ be the class of all Borel subsets in $S\subset {\mathbb{C}}_{\infty}$. We denote by ${\sf Meas}(S)$  the class of all Borel signed measures, or, \textit{charges\/} on a set $S\in {\text{\sf Bor}}  ({\mathbb{C}}_{\infty})$, and    
${\sf Meas}_{{\rm cmp}}(S)$ is the class of charges $\mu \in {\sf Meas}(S)$ with a compact {\it support\/} ${\rm supp\,}  \mu \Subset S$; 
${\sf Meas}^+(S):=\{ \mu\in  {\sf Meas} (S)\colon \mu\geq 0\}$ is the subclass of all positive {\it measures};
${\sf Meas}_{{\rm cmp}}^+(S):= {\sf Meas}_{{\rm cmp}} (S)\cap {\sf Meas}^+(S)$;
${\sf Meas}^{1+}(S):=\{\mu \in {\sf Meas}^+(S)\colon \mu(S)=1 \}$ 
is the class of {\it probability measures\/} on ${\text{\sf Bor}} (S)$.  
Besides, ${\sf Meas}:={\sf Meas}({\mathbb{C}})$, ${\sf Meas}_{{\rm cmp}}:={\sf Meas}_{{\rm cmp}}({\mathbb{C}})$, ${\sf Meas}^+:={\sf Meas}({\mathbb{C}})$,  etc. For $S\in \text{\sf Bor}(\mathbb C)$ and a charge $\nu \in {\sf Meas}(S)$, we let
$\nu^+\in {\sf Meas}^+(S)$, $\nu^-:=(-\nu)^+\in {\sf Meas}^+(S)$,
$|\nu| := \nu^+ +\nu^-\in {\sf Meas}^+(S)$ respectively denote {\it upper, lower,\/} and {\it total variations\/} of $\nu$.

\begin{definition}[{\rm see  \cite{Landkoff},  \cite{KhaRozKha19}, 
\cite{Gamelin}, \cite{Kha03}, cf. \cite[Definition 2]{Kha20}}]\label{Def2}  For  $\delta, \omega \in {\sf Meas}^+$ and
$H\subset {\overline {\mathbb{R}}}^{{\mathbb{C}}}$, let us assume that the integrals 
$\int h {\,{\rm d}}{\delta}$ and  $\int h {\,{\rm d}}{\omega}$ are well defined with values in ${\overline {\mathbb{R}}}$ for each function $h\in H$. We say that the measure  $\omega$ is a {\it balayage,\/} or, sweeping (out), of the measure ${\delta}$ {\it with respect to\/} $H$, or, briefly,  $\omega$  is a {\it $H$-balayage of ${\delta}$,\/} if 
\begin{equation}\label{balnumu}
	\int h {\,{\rm d}} {\delta} \leq \int h{\,{\rm d}} \omega \in \overline {\mathbb{R}} \quad\text{\it for each\/ $h\in H$.}
	\end{equation} 
	Obviously, \eqref{balnumu} implies the equality
	\begin{equation}\label{balnumu=}
	\int h {\,{\rm d}} {\delta} =\int h{\,{\rm d}} \omega \in {\mathbb{R}} \quad\text{\it if\/ $\pm h\in H$
		and } \int h{\,{\rm d}} \delta \neq -\infty. 
	\end{equation} 
\end{definition}

For  $S \subset \mathbb C$, by ${\sf har} (S)$ and ${\sf sbh}(S)$
we denote the classes of all {\it harmonic\/}  and    {\it subharmonic\/}  
functions on open neighborhoods  of $S$, respectively; ${\sf sbh}_*(S):=\{u\in {\sf sbh}(S)\colon u\not\equiv -\infty\}$. 
Widely used variants of classes   $H\subset  {\overline {\mathbb{R}}}^{{\mathbb{C}}}$ from Definition \ref{Def2}  are
${\sf har}:={\sf har}({\mathbb{C}} )$,  ${\sf sbh}:={\sf sbh}({\mathbb{C}} )$, ${\sf sbh}_*:={\sf sbh}_*({\mathbb{C}})$ \cite{Gamelin}, \cite{KhaRozKha19}.   
For an open set $O\subset {\mathbb{C}}$, the {\it  Riesz measure of\/} $u\in {\sf sbh}_*( O)$ is a positive Borel measure 
\begin{equation}\label{df:cm}
\varDelta_u:= \frac{1}{2\pi} {\bigtriangleup}{u}\in {\sf Meas}^+(  O),
\end{equation}
where $\bigtriangleup$ is  the {\it Laplace operator\/}  acting in the sense of the
theory of distributions or generalized functions.

For  $p\in \overline {\mathbb{R}}$ and ${\mathbb{Z}}:=(-{\mathbb{N}})\cup {\mathbb{N}}_0$, we set 
\begin{subequations}\label{fc}
\begin{align}
\lfloor p \rfloor&:=:
\text{floor} (p):=  \sup\{k\in {\mathbb{Z}}\colon  k\leq p\} \in {\mathbb{Z}}\cup \pm\infty 
\tag{\ref{fc}f}\label{{fc}f}\\
\intertext{(the {\it integer part\/} of $p$ if $p\in {\mathbb{R}}$), and} 
\lceil p \rceil &:=:
\text{ceil} (p):=  \inf\{k\in {\mathbb{Z}} \colon  k\geq p\} \in {\mathbb{Z}}\cup \pm\infty .
\tag{\ref{fc}c}\label{{fc}c}
\end{align}
\end{subequations}

Let  $p\in \overline {\mathbb{R}}^+$, let 
\begin{equation}\label{Monp}
{\sf Mon}_p:=\Bigl\{ z\underset{\text{\tiny $z\in {\mathbb{C}}$}}{\longmapsto}  cz^k \colon p\geq k\in {\mathbb{N}}_0, c\in {\mathbb{C}}  \Bigr\}\overset{\eqref{{fc}f}}{=}{\sf Mon}_{\lfloor p\rfloor} \subset {\mathbb{C}}^{{\mathbb{C}}} 
\end{equation}
be the class of  {\it complex monomials\/}  of degree  at most $p$, and let 
\begin{equation}\label{{mon}P}
{\sf Pol}_p:=\sum_{k=0}^{\lfloor p\rfloor}   {\sf Mon}_k 
\end{equation}
be the class of  all {\it complex  polynomials\/} $P\underset{\text{\tiny $z\in {\mathbb{C}}$}}{\in} {\mathbb{C}}[z]$  of \textit{degree\/} $\deg P\leq p$.

In this article, we often consider the following two classes as a class $H\subset \overline {\mathbb{R}}^{{\mathbb{C}}}$ from Definition \ref{Def2}. 
The first class 
\begin{equation}\label{{mon}m}
{\sf mon}_p:={\rm Re\,} {\sf Mon}_p={\rm Im\,} {\sf Mon}_p={\sf mon}_{\lfloor p\rfloor}\subset {\sf har}, 
\quad {\sf mon}_{\infty}:={\sf mon}_{+\infty},
\end{equation}
consists of harmonic real homogeneous polynomials of degree $\leq p$.
The real \textit{linear span\/} of 
of this class ${\sf mon}_p$   coincides with the real space of all harmonic polynomials of degree at most $p$.

The second class is the union 
\begin{equation}\label{{mon}l}
{\sf lnmon}_p:=\ln_{{\mathbb{C}}}\cup {\sf mon}_p\subset {\sf sbh}_*  , \quad\text{where } \ln_{{\mathbb{C}}}:=\Bigl\{z \underset{\text{\tiny $z\in {\mathbb{C}}$}}{\longmapsto} \ln |z-w|\colon  w\in {\mathbb{C}} \Bigr\}
\end{equation}
is the class  of \textit{logarithmic functions\/} generated by the 
\textit{logarithmic kernel\/}
\begin{equation*}
\ln|\cdot-\cdot| \colon  (z,w)\longmapsto 
\begin{cases}
\ln |w-z| \quad &\text{if $z\neq w$},\\
-\infty \quad &\text{if $z= w$}, 
\end{cases} \qquad  (z,w)\in{\mathbb{C}}^2.
\end{equation*}

We denote by 
$D(z,t):=\bigl\{z'\in {\mathbb{C}}\colon |z'-z|< t\bigr\}$, $\overline D(z,t):=\bigl\{z'\in {\mathbb{C}}\colon |z'-z|\leq  t\bigr\}$, $\partial \overline D(z,t):=\overline D(z,t)\!\setminus\!  D(z,t)$  an {\it open disk,\/} a {\it closed disk,\/} a {\it circle of radius $t\in \overline {\mathbb{R}}^+$ centered at $z\in {\mathbb{C}}$}, respectively; $D(t):=D(0,t)$, $\overline D(t):=\overline D(0,t)$,  $\partial \overline D(t):=\partial \overline D(0,t)$, and ${\mathbb{D}}:=D(1)$, $\overline {\mathbb{D}}:=\overline D(1)$, $\partial {\mathbb{D}}=\partial \overline {\mathbb{D}}:=\partial \overline D(1)$  are the {\it open unit disk,\/} the {\it closed unit disk,\/} the {\it  unit circle,\/} respectively.

For $\nu\in {\sf Meas}$, $z\in {\mathbb{C}}$, and $0\leq r< R\in \overline {\mathbb{R}}^+$, we set  
\begin{subequations}\label{mB}
\begin{align}
\nu(z,r)&:=\nu\bigl(\overline D(z,r)\bigr), \quad 
\nu^{{\text{\tiny\rm rad}}}(r):=\nu(0,r), 
\tag{\ref{mB}$\nu$}\label{{mB}nu}\\
  {\sf N}_{\nu} (z, r,R)&:=\int_{r}^R \frac{\nu(z, t)}{t}{\,{\rm d}} t , \quad {\sf N}_{\nu}^{{\text{\tiny\rm rad}}} (r,R)
:= {\sf N}_{\nu} (0, r,R), 
\tag{\ref{mB}N}\label{{mB}N}
\\
\intertext{and also, for a function $t\underset{\text{\tiny $t\in [r,R]$}}{\longmapsto} f(t)\in {\mathbb{R}}$,  provided $\nu({\mathbb{C}})=\nu^{{\text{\tiny\rm rad}}}(+\infty)\in {\mathbb{R}}$,}
{\sf N}^{\star}_{\nu}\bigl(r,R; t\mapsto f(t)\bigr)&:=\int_{r}^{R}\frac{\nu^{{\text{\tiny\rm rad}}}(+\infty)-\nu^{{\text{\tiny\rm rad}}}(t)}{t}\, f(t){\,{\rm d}} t.
\tag{\ref{mB}$\star$}\label{{mB}star}
\end{align}
\end{subequations}

We list some elementary  properties of balayage with respect to certain classes from  \eqref{{mon}m} and \eqref{{mon}l}  and the classes ${\sf har}$ and ${\sf sbh}$.
\begin{propos}\label{pr:bal} 
	Let  $p\in {\mathbb{R}}^+$,  $\delta\in {\sf Meas}^+$, $\omega\in {\sf Meas}^+$, and 
	\begin{equation}\label{tmf}
	0<\delta ({\mathbb{C}})=\delta^{{\text{\tiny\rm rad}}}(+\infty)<+\infty, \quad   0<\omega ({\mathbb{C}})=\omega^{{\text{\tiny\rm rad}}}(+\infty)<+\infty.
	\end{equation}  
	\begin{enumerate}[{\rm [{b}1]}]
	\item\label{b2} 
		A measure $\omega\in {\sf Meas}^+$ is a ${\sf mon}_p$-balayage of  
		$\delta\in {\sf Meas}^+$  for $p<1$ if and only if $\omega({\mathbb{C}})=\delta ({\mathbb{C}})$.  
		\item\label{b3} 
		The following tree statements are equivalent: 
		\begin{enumerate}[{\rm (i)}]
	\item\label{b3i} 	$\omega$ is a ${\sf mon}_p$-balayage of  		$\delta$; 
\item\label{b3ii} $\int_{\text{\tiny ${\mathbb{C}}$}} z^k{\,{\rm d}} \delta(z)\underset{\text{\tiny $k\in {\mathbb{N}}_0$}}{=}\int_{\text{\tiny ${\mathbb{C}}$}} z^k{\,{\rm d}} \omega(z)$ for each $k\leq p$;
\item\label{b3iii} 
equality  \eqref{balnumu=} 
is fulfilled for each polynomial $h\overset{\eqref{{mon}P}}{\in} {\sf Pol}_p$.  
\end{enumerate}
\item\label{b5} The integral $\int h{\,{\rm d}} \omega\in {\mathbb{R}}\cup{-\infty}$  are well defined 
for each  function $h\overset{\eqref{{mon}l}}{\in} {\sf lnmon}_p$ under the condition\/  ${\sf N}^{\star}_{\nu}\bigl(1,+\infty; t\mapsto t^{\lfloor p \rfloor}\bigr)\overset{\eqref{{mB}star}}{<}+\infty$. 
	\end{enumerate}
	Suppose that, in addition to \eqref{tmf}, $\delta\in {\sf Meas}_{{\rm cmp}}^+$ and $ \omega\in {\sf Meas}_{{\rm cmp}}^+$.
	\begin{enumerate}
		\item[{\rm [{b}4]}]\label{b7} 
		If $\omega$ is a ${\sf mon}_{\infty}$-balayage of $\delta$, then
		$\omega$ is a ${\sf har}$-balayage of $\delta$.
		\item[{\rm [{b}5]}]\label{b8} 
		If $\omega$  is a ${\sf lnmon}_{\infty}$-balayage of $\delta$, then
		$\omega$ is a  ${\sf sbh}$-balayage of $\delta$.
	\end{enumerate} 
	\end{propos}
We omit simple proofs of statements  [{b}1]--[{b}5] of Proposition \ref{pr:bal}
and note only that statements [{b}4] and [{b}5] are implicitly proved in \cite[\S~1]{Kha91}. 
Additional general  properties of balayage can be found in \cite{Kha20}.

\section{Logarithmic potentials,  balayage, and polar sets}

\begin{definition}\label{df:pot}  {\rm \cite{R}, \cite{Arsove}, \cite{Arsove53p}, 
		\cite[Definition 2]{Kha03}, \cite[3.1, 3.2]{KhaRoz18}}
	Let $\nu \in {\sf Meas}$ be a charge such that  ${\sf N}^{\star}_{|\nu|}(1,+\infty; t\mapsto 1)<+\infty$, i.e., 
	\begin{equation*}
	\int_1^{+\infty}  \frac{|\nu|^{{\text{\tiny\rm rad}}} (+\infty)-|\nu|^{{\text{\tiny\rm rad}}} (t)}{t} {\,{\rm d}} t <+\infty, 
	\end{equation*}

Its {\it logarithmic potential\/}  
	\begin{equation}\label{pmu}
		{{\sf pt}}_{\nu}(z):=\int \ln |w-z| {\,{\rm d}} \nu (w), 
			\end{equation}
is uniquely determined on \cite{Arsove}, \cite[3.1]{KhaRoz18}
\begin{equation}\label{Dom}
{\sf Dom}\, {\sf pt}_{\nu} \overset{\eqref{{mB}N}}{:=}
\Bigl\{z\in {\mathbb{C}}\colon
\min\bigl\{ {\sf N}_{\nu^-}(z,0,1), {\sf N}_{\nu^+}(z,0,1) \bigr\}<+ \infty  \Bigr\}
\end{equation}
by values in  $\overline {\mathbb{R}}$, 
and  the set $E:=(\complement \, {\sf Dom}\, {\sf pt}_{\omega})\!\setminus\!\infty$ is {\it polar  $G_{\delta}$-set\/} with zero {\it outer capacity\/} 
$\text{\rm Cap}^*(E)=0$. Evidently, ${\sf pt}_{\nu}\in {\sf har}\bigl({\mathbb{C}}\!\setminus\!{\rm supp\,}  |\nu|\bigr)$. If $\nu \in {\sf Meas}^+_{{\rm cmp}}$, then ${\sf pt}_{\nu}\in {\sf sbh}_*$, and  $\varDelta_{{\sf pt}_{\nu}}\overset{\eqref{df:cm}}{=}\nu\in {\sf Meas}^+$. 
\end{definition}

\begin{theorem}\label{Pr_pol}
	Let $E\in \text{\sf Bor}(\mathbb C)$ be a polar set, and let $\omega \in {\sf Meas}^+$ be a $\ln_{{\mathbb{C}}}$-balayage of  $\delta \in {\sf Meas}^+$  provided\/  ${\sf N}^{\star}_{\delta+\omega}(1,+\infty; t\mapsto 1)<+\infty$. 
	\begin{enumerate}[{\rm {\bf p}1.}]
		\item\label{p1}  If  $\int {{\sf pt}}_{\nu} {\,{\rm d}} \delta \in {\mathbb{R}}$  
	for each ${{\sf pt}}_{\nu}\not\equiv -\infty$ with $\nu \in {\sf Meas}_{{\rm cmp}}^+$, then $\omega (E)=0$. 
		\item\label{p2}  If  ${\delta} \in {\sf Meas}_{{\rm cmp}}^+$, 
then  $\omega (E\!\setminus\!{\rm supp\,}  {\delta})=0$. 
	\end{enumerate}
\end{theorem}
\begin{proof} We can consider only  $E\Subset {\mathbb{C}}$. 
	There is a potential ${{\sf pt}}_{\nu}$ with $\nu \in {\sf Meas}_{{\rm cmp}}^+$ such that this set $E$ is included in the   minus-infinity $G_{\delta}$-set  $E_{\nu}:=(-\infty)_{{{\sf pt}}_{\nu}}
	:=\bigl\{z\in {\mathbb{C}}\colon {\sf pt}_{\nu}(z)=-\infty\bigr\}\Subset {\mathbb{C}}$.
	
	{\bf p}\ref{p1}. The condition  $\int {{\sf pt}}_{\nu} {\,{\rm d}} \delta >-\infty$ entails
	\begin{equation*}
	-\infty <\int {{\sf pt}}_{\nu} {\,{\rm d}} \delta \overset{\eqref{balnumu}}{\leq}  
	\int {{\sf pt}}_{\nu} {\,{\rm d}} \omega 
	=\left(\int_{E_{\nu}}+\int_{{\mathbb{C}}\!\setminus\!E_{\nu}}\right)
	{{\sf pt}}_{\nu} {\,{\rm d}} \omega 	=(-\infty) \cdot \omega (E_{\nu})+\int_{{\mathbb{C}}\!\setminus\!E_{\nu}}
	{{\sf pt}}_{\nu} {\,{\rm d}} \omega. 
	\end{equation*}
	This is only possible when  $\omega (E_{\nu}) =0$.
	
	{\bf p}\ref{p2}. For any $k\in {\mathbb{N}}$ 
	there exists an finite cover of ${\rm supp\,}  {\delta}\Subset {\mathbb{C}}$ by disks $D(x_j,1/k)\Subset {\mathbb{C}}$ such that the open sets 
	\begin{equation*}
	O_k:=\bigcup_j D(x_j,1/k)\Subset {\mathbb{C}},\quad
	{\rm supp\,}  {\delta} \Subset O_{k+1}\underset{\text{\tiny $k\in {\mathbb{N}}$}}{\subset} O_{k},  \quad {\rm supp\,}  {\delta} =\bigcap_{k\in {\mathbb{N}}} O_k,
	\end{equation*}
	have complements $\complement O_k$ \textit{without isolated points.\/} Then 
	every open subset  $O_k\Subset {\mathbb{C}}$ is regular for the Dirichlet problem.
	It suffices to prove that the equality $\omega (E_\nu\!\setminus\!O_k)=0$ holds for each 
	$k\in {\mathbb{N}}$. Consider the functions 
	\begin{equation*}\label{Uk}
	u_k=\begin{cases}
	{{\sf pt}}_{\nu} \text{ \it  on ${\mathbb{C}} \!\setminus\!O_k$},\\
	\text{\it the harmonic extension of ${{\sf pt}}_{\nu}$ from $\partial O_k$ into $O_k$}\text{ on $O_k$},
	\end{cases}   
	 k\in {\mathbb{N}}.
	\end{equation*}
	We have  $u_k\in {\sf sbh}_*$, and $u_k$ is bounded from below in ${\rm supp\,}  {\delta} \Subset O_k$. Hence
	\begin{multline*}
	-\infty <\int u_k {\,{\rm d}} {\delta}
	\overset{\eqref{balnumu}}{\leq}
	\int u_k {\,{\rm d}} \omega=
	\left(\int_{{\mathbb{C}}\!\setminus\!(E_{\nu}\!\setminus\!O_k)}+\int_{E_{\nu}\!\setminus\!O_k}\right) u_k {\,{\rm d}} \omega
	\\
	\leq {\rm const}_{\nu} \int_{{\mathbb{C}}} \log (2+|z|) {\,{\rm d}} \omega (z)+(-\infty)\cdot \omega(E_{\nu}\!\setminus\!O_k)
	\leq {\rm const}_{\nu, \omega}+(-\infty)\cdot \omega(E_{\nu}\!\setminus\!O_k),
	\end{multline*}
where here and further we denote by ${\rm const}_{a_1,a_2, \dots}$ constants that depend only 
	on the parameters-indexes $a_1,a_2, \dots$; ${\rm const}^+_{a_1,a_2, \dots}\in \mathbb R^+$.
	This is only possible when   $\omega(E_{\nu}\!\setminus\!O_k)=0$.
\end{proof}

\begin{remark}\label{rem:contr} Theorem \ref{Pr_pol} is not true for ${\sf mon}_p$-balayage\/ {\rm (see \cite[Example]{MenKha19})}. 
\end{remark}

\section{Duality for ${\sf mon}_p$-balayage  and ${\sf lnmon}_p$-balayage}

For numbers $r_0\in {\mathbb{R}}^+ \!\setminus\! 0$ and $q\in {\mathbb{N}}_0$, we consider the classical  subharmonic {\it Weierstrass\,--\,Hadamard kernel 
of genus\/} $q$ \cite[3]{Arsove53p}, \cite[(3.2)]{Kha09}, \cite[\S~2, Example 3($E_q$)]{Kha07},
\cite{Bergweiler}, \cite{Gil}, \cite{Hansmann}, \cite{HK}, \cite{Merzlyakov}
\begin{subequations} \label{kern:WA}
\begin{align}
k_q(w , z)&:=\log |w-z| 
\text{ if }(w , z) \in r_0 {\mathbb{D}} \times {\mathbb{C}},
\tag{\ref{kern:WA}$_0$}\label{{kern:WA}0}
\\
k_q(w , z)&:=\log \Bigl|1-\frac{z}{w}\Bigr|+\sum\limits_{k=1}^q  {\rm Re} \frac{z^k}{kw^k} 
\text{ if }(w , z) \in ({\mathbb{C}}\!\setminus\! r_0{\mathbb{D}}) \times {\mathbb{C}},
\tag{\ref{kern:WA}$_\infty$}\label{{kern:WA}i}
\end{align}
\end{subequations} 
where $\sum_{k=1}^{0}\dots :=0$,  $r_0$ is indicated only when necessary, and the Riesz measure $\varDelta_{k_q(w, \cdot)}
\overset{\eqref{df:cm}}\in {\sf Meas}^{1+}_{{\rm cmp}}$
is the {\it Dirac measure\/} at $w$ with  ${\rm supp\,}  \varDelta_{k_q(w, \cdot)} =w$  for each  $w\in {\mathbb{C}}$.

\begin{propos}\label{pr2} 
If $\delta, \omega\in {\sf Meas}^+$ 
are measures such that ${\sf N}^{\star}_{\delta+\omega}(1,+\infty; t\mapsto t^{\lfloor p\rfloor})\overset{\eqref{{mB}star}}{<}+\infty$ for $p\in {\mathbb{R}}^+$, and $\omega$
is a ${\sf mon}_p$-balayage  of $\delta$,   then  
\begin{equation}\label{kpmt}
\int_{{\mathbb{C}}} k_{\lfloor p\rfloor} (w,z){\,{\rm d}} (\omega-\delta)(z) ={\sf pt}_{\omega-\delta}(w)\quad \text{for each $w\overset{\eqref{Dom}}{\in} {\sf Dom}\, {\sf pt}_{\omega-\delta}$.}
\end{equation}
 If $\omega$ is a ${\sf lnmon}_p$-balayage  of $\delta$,   then
\begin{equation}\label{kpmtl}
\int_{{\mathbb{C}}} k_{\lfloor p\rfloor} (w,z){\,{\rm d}} \delta(z)\leq  \int_{{\mathbb{C}}} k_{\lfloor p\rfloor} (w,z){\,{\rm d}} \omega(z)\quad \text{for each $w\in {\mathbb{C}}$}.
\end{equation}
\end{propos}
\begin{proof} If $w\overset{\eqref{{kern:WA}0}}{\in} r_0 {\mathbb{D}}$, then  \eqref{kpmt}  is obvious.
If $w\overset{\eqref{{kern:WA}i}}{\in} {\mathbb{C}}\!\setminus\!r_0 {\mathbb{D}}$, then, by statements [b\ref{b3}] and [b\ref{b5}] of Proposition \ref{pr:bal}, the integration of  
$k_{\lfloor p\rfloor}$ from \eqref{kern:WA} with respect to the charge $\omega-\delta\in {\sf Meas}$
gives \eqref{kpmt}. 

If $\omega$ is a ${\sf lnmon}_p$-balayage  of $\delta$, then, using definition \eqref{kern:WA} of the Weierstrass\,--\,Hadamard kernel 
of genus $\lfloor p\rfloor$, we obtain \eqref{kpmtl}
by Definition \ref{Def2} with \eqref{balnumu} for $h\overset{\eqref{{mon}l}}{\in}  \ln_{{\mathbb{C}}}$ and Proposition   \ref{pr:bal} with [b\ref{b2}]--[b\ref{b3}].
\end{proof}

For $r\in {\mathbb{R}}^+$ and  $u\colon \partial \overline D(z, r)\to {\overline {\mathbb{R}}}$ we define
\begin{equation}\label{df:MCBM}
{\sf M}_u(z,r):=\sup_{\partial D(z,r)}u,\quad {\sf M}_u^{{\text{\tiny\rm rad}}}(r):={\sf M}_u(0,r). 
\end{equation}
If $u\in {\sf sbh}$, then ${\sf M}_u(z,r)=\sup\limits_{0\leq t\leq r}{\sf M}_u(z,t)$, and  ${\sf M}_u^{{\text{\tiny\rm rad}}}(r)=\sup\limits_{0\leq t\leq r}{\sf M}_u^{{\text{\tiny\rm rad}}}(t)$.

An elementary consequence  the classical Borel\,--\,Carath\'eodory inequality for disks $\overline  D(2r)$ is 

\begin{propos}\label{prBC}
If $f$ is an entire function on ${\mathbb{C}}$, then 
\begin{equation*}
{\sf M}_{|f|}(r)\leq 2 {\sf M}_{{\rm Re}\, f}(2r)+3\bigl|f(0)\bigr|\quad\text{for each $r\in {\mathbb{R}}^+$}.
\end{equation*}  
\end{propos}

\begin{theorem}\label{DT1}
If  $\omega\in {\sf Meas}_{{\rm cmp}}^+$ is a\/  ${\sf mon}_p$-balayage of
	$\delta\in {\sf Meas}_{{\rm cmp}}^+$,  then 
	\begin{subequations}\label{pmu0}
		\begin{align}
	{\sf pt}_{\delta}\in {\sf sbh}_*\cap 
		{\sf har}({\mathbb{C}} \!\setminus\! {\rm supp\,}  \delta),&	\quad 	{\sf pt}_{\omega}\in {\sf sbh}_*\cap 
		{\sf har}({\mathbb{C}} \!\setminus\! {\rm supp\,}  \omega),
		\tag{\ref{pmu0}p}\label{{pmu0}p}\\
		{\sf pt}_{\omega}(w)=  {{\sf pt}}_{\delta}(w)&+O\Bigl(\frac{1}{|w|^{\lfloor p\rfloor+1}}\Bigr) \quad 
\text{as }w\to \infty .
		\tag{\ref{pmu0}O}\label{{pmu0}o}
		\end{align}
	\end{subequations}
If   $\omega$ is a ${\sf lnmon}_p$-balayage of ${\delta}$, then, in addition to\/  \eqref{pmu0}, we have 
	\begin{equation}\label{pmu0+}
	{\sf pt}_{\omega}\geq {{\sf pt}}_{\delta} \quad \text{on ${\mathbb{C}}$.}
	\end{equation}
	\underline{Conversely}, suppose that there are a set $S\Subset {\mathbb{C}}$,
	and  	functions  
	\begin{subequations}\label{p}
		\begin{align}
		d\in {\sf sbh} \cap  {\sf har} ({\mathbb{C}}\!\setminus\! S),& \quad v\in {\sf sbh} \cap  {\sf har} ({\mathbb{C}}\!\setminus\! S) 
		\quad\text{\rm  (cf. \eqref{{pmu0}p})}
		\tag{\ref{p}p}\label{{pmu0+}p}
		\\
\intertext{such that}
		v(w)= d(w)&+O\Bigl(\frac{1}{|w|^{\lfloor p\rfloor+1}}\Bigr) \quad\text{as } w\to \infty \quad \text{\rm  (cf. \eqref{{pmu0}o})}
		\tag{\ref{p}O}\label{{pmu0+}o}.
		\end{align}
	\end{subequations}
	Then the  Riesz measure
	\begin{equation}\label{mu}
	\omega:=\varDelta_v\overset{\eqref{df:cm}}{:=}\frac{1}{2\pi}\bigtriangleup\! v \overset{\eqref{p}}{\in} {\sf Meas}^+({\rm clos\,} S)\subset {\sf Meas}_{{\rm cmp}}^+
	\end{equation} 
	of $v$ is a ${\sf mon}_p$-balayage of the Riesz measure
	\begin{equation}\label{d}
	\delta:=\varDelta_d\overset{\eqref{df:cm}}{:=}\frac{1}{2\pi}\bigtriangleup\! d \overset{\eqref{p}}{\in} {\sf Meas}^+({\rm clos\,} S)\subset {\sf Meas}_{{\rm cmp}}^+
	\end{equation} 
of $d$. If, in addition to \eqref{p}, we have $v\geq d$ on ${\mathbb{C}}$ {\rm (cf.  \eqref{pmu0+}),} then $\omega$ is a ${\sf lnmon}_p$-balayage of $\delta$.
\end{theorem}
\begin{proof}  The first  property \eqref{{pmu0}p} for potentials ${\sf pt}_\delta$  and ${\sf pt}_\omega$ with compact supports 
${\rm supp}\, \delta, {\rm supp\,}  \omega\Subset {\mathbb{C}}$  is  obvious. 
Let's prove property  \eqref{{pmu0}o}.  By Proposition \ref{pr2} we have \eqref{kpmt}.
If \begin{equation}\label{w}
|w|\overset{\eqref{kern:WA}}{>}
R_0:=\max \{r_0, 2s\},  \text{ where $s:=\sup\bigl\{|z|\colon z\in {\rm supp\,}  (\delta+\omega) \bigr\}$},
\end{equation} 
then  $w\overset{\eqref{Dom}}{\in} {\sf Dom}\, {\sf pt}_{\omega-\delta}$, and 
we can use the Taylor series expansion for the integrand expression  $k_{\lfloor p\rfloor}$ of integral from the left-hand side of \eqref{kpmt} in the form 
\begin{multline*}
\bigl|{\sf pt}_{\omega-\delta}(w)\bigr|\overset{\eqref{kpmt}}{=}\Bigl|\int_{{\mathbb{C}}} k_{\lfloor p\rfloor} (w,z){\,{\rm d}} (\omega-\delta)(z) \Bigr|\overset{\eqref{{kern:WA}i}}{=}\biggl| \int \biggl(\sum_{k=\lfloor p\rfloor+1}^{\infty} -{\rm Re} \frac{z^k}{kw^k}\biggr) {\,{\rm d}} (\omega-\delta)(z) \biggr|
\\ \leq  \int \sum_{k=\lfloor p\rfloor+1}^{\infty} \Bigl|-{\rm Re} \frac{z^k}{kw^k}\Bigr| {\,{\rm d}} (\omega+\delta)(z) 
\overset{\eqref{w}}{\leq} \int_{\overline D(s)} \frac{|z|^{\lfloor p\rfloor+1}}{|w|^{\lfloor p\rfloor+1}}
\sum_{m=0}^{\infty} \Bigl|\frac{z}{w}\Bigr|^m{\,{\rm d}} (\omega+\delta)(z)
\\
\overset{\eqref{w}}{\leq} \frac{2}{|w|^{\lfloor p\rfloor+1}}\int_{\overline D(s)}
|z|^{\lfloor p\rfloor+1}{\,{\rm d}} (\omega+\delta)(z) 
=O\Bigl(\frac{1}{|w|^{\lfloor p\rfloor+1}}\Bigr)\quad \text{as }w\to\infty.
\end{multline*}
The latter means \eqref{{pmu0}o}. Finally, if $\omega$ is a $\ln_{{\mathbb{C}}}$-balayage of $\delta$, then by inequalities \eqref{balnumu} for $h\in \ln_{{\mathbb{C}}}$ in Definition \ref{Def2} we obtain ${\sf pt}_{\delta}\leq {\sf pt}_{\omega}$ on ${\mathbb{C}}$.

\underline{Conversely,} \eqref{{pmu0+}p} implies \eqref{mu} and \eqref{d}. 
First, we prove that $v$ is the  logarithmic potential ${\sf pt}_{\omega}$ under the assumption that the function $d$ is the logarithmic potential
\begin{equation}\label{pd}
d(w)\underset{\text{\tiny $w\in {\mathbb{C}}$}}{\equiv} {\sf pt}_{\delta} (w)=O\big(\ln |w|\big)\quad \text{as $w\to \infty$.}
\end{equation}
 It follows from \eqref{pd} and   \eqref{{pmu0+}o} that  $v(w)=O\big(\ln |w|\big)$ as $w\to \infty$. By the Weierstrass\,--\,Hadamard representation theorem, such subharmonic functions with the Riesz measure $\omega\in {\sf Meas}^+$ can be represented in the form $v:={\sf pt}_{\omega}+C$, where $C$ is a constant. 
By \cite[Theorem 3.1.2]{R},  we obtain
\begin{equation}\label{pt2}
\begin{cases}
{\sf pt}_{\omega}(w)=\omega({\mathbb{C}})\log |w|+O\bigl(1/|w|\bigr),\\
{\sf pt}_{\delta}(w)=\delta ({\mathbb{C}})\log |w|+O\bigl(1/|w|\bigr)
\end{cases}
\text{ as }w\to \infty.
\end{equation}
Hence we have
\begin{multline*}
C\underset{\text{\tiny $w\in {\mathbb{C}}$}}{\equiv} v(w)-{\sf pt}_{\omega}(w)\underset{\text{\tiny $w\to \infty$}}{\overset{\eqref{{pmu0+}o}}{=}}{\sf pt}_{\delta}(w)+O\Bigl(\frac{1}{|w|^{\lfloor p\rfloor+1}}\Bigr) -{\sf pt}_{\omega}(w)\\
\underset{\text{\tiny $w\to \infty$}}{\overset{\eqref{pt2}}{=}}\bigl(\delta ({\mathbb{C}})-\omega({\mathbb{C}})\bigr)
\ln |w|+O\Bigl(\frac{1}{|w|^{\lfloor p\rfloor+1}}\Bigr)+O\Bigl(\frac1{|w|}\Bigr)\quad \text{ as }w\to \infty. 
\end{multline*}
This is only possible if $\delta ({\mathbb{C}})-\omega({\mathbb{C}})=0$ and $C=0$. Thus, we have
\begin{equation}\label{Ow}
v={\sf pt}_{\omega}, \quad \omega({\mathbb{C}})=\delta ({\mathbb{C}}), \quad 
{\sf pt}_{\omega-\delta}(w)\underset{\text{\tiny $w\to \infty$}}{\overset{\eqref{{pmu0+}o}}{=}}O\Bigl(\frac{1}{|w|^{\lfloor p\rfloor+1}}\Bigr).
\end{equation}
In particular, from here we have
\begin{equation}\label{1w}
{\sf pt}_{\omega-\delta}(z)= \int_{{\mathbb{C}}} \ln \Bigl| 1-\frac{z}{w}\Bigr| {\,{\rm d}} (\omega-\delta)(z)
\quad\text{for each $w\overset{\eqref{Dom}}{\in} {\sf Dom}\, {\sf pt}_{\omega-\delta}$,}
\end{equation}
and, by Proposition \ref{pr:bal}[b\ref{b2}],   $\omega$ is a ${\sf mon}_p$-balayage of $\delta$
in the case $p<1$. Let's prove  that this measure $\omega\in {\sf Meas}_{{\rm cmp}}^+$ is a ${\sf mon}_p$-balayage of $\delta$
in the case $p\geq 1$. 
By expanding in the Taylor series of the corresponding analytic branch of the function $z\underset{\text{\tiny $z\in {\mathbb{D}}$}}{\longmapsto} \ln (1-z)$, 
we obtain the following representation 
\begin{subequations}\label{ptp}
\begin{align}
{\sf pt}_{\omega-\delta}(w)\overset{\eqref{1w}}{=}\int_{\overline D(s)} \sum_{k=1}^{\lfloor p\rfloor}{\rm Re} \frac{-z^k}{kw^k}
{\,{\rm d}} (\omega-\delta)(z)  
= {\tt Q}_{\lfloor p \rfloor}(w)+ {\tt R}_{\lfloor p \rfloor} (w) 
\tag{\ref{ptp}p}\label{{ptp}p}
\\
\text{if $w\overset{\eqref{w}}{\in} \complement D(R_0)$, where ${\rm supp\,} (\delta+\omega)\overset{\eqref{w}}{\subset} \overline D(s)$, i.e., } \Bigl|\frac{z}{w}\Bigr|\leq \frac{1}{2},
\tag{\ref{ptp}{\it w}}\label{{ptp}w}\\
{\tt Q}_{\lfloor p \rfloor}(w)\underset{\text{\tiny $|w|\geq R_0$}}{\equiv} {\rm Re} \sum_{k=1}^{\lfloor p\rfloor}\frac{q_k}{w^k} \quad 
\text{with } q_k:=-\frac{1}{k}\int_{\overline D(s)} z^k {\,{\rm d}} (\omega-\delta)(z)
\tag{\ref{ptp}{\tt Q}}\label{{ptp}Q}
\end{align}
\end{subequations}
is harmonic rational function on $\mathbb C_{\infty}\!\setminus\!0$ with ${\tt Q}_{\lfloor p \rfloor}(\infty)=0$,  and, for $|w|\overset{\eqref{{ptp}w}}{\geq} R_0$,  
\begin{equation*}
 \bigl|{\tt R}_{\lfloor p \rfloor} (w)\bigr|=\biggl|\int_{\overline D(s)} \sum_{\lfloor p\rfloor+1}^{\infty}{\rm Re} \frac{-z^k}{kw^k}
{\,{\rm d}} (\omega-\delta)(z) \biggr|\overset{\eqref{{ptp}w}}{\leq} \frac{2s^{\lfloor p\rfloor+1}}{|w|^{\lfloor p\rfloor+1}}
(\omega+\delta)({\mathbb{C}}).
 \end{equation*}
Hence, in view of \eqref{Ow} and \eqref{ptp}, we have
\begin{equation}\label{tQ}
{\tt Q}_{\lfloor p \rfloor}(w)\overset{\eqref{{ptp}Q}}{=}
{\rm Re} \sum_{k=1}^{\lfloor p\rfloor}\frac{q_k}{w^k}\overset{\eqref{{ptp}p}}{=}
O\Bigl(\frac{1}{|w|^{\lfloor p\rfloor+1}}\Bigr)\quad\text{as }w\to \infty.
\end{equation}
Therefore,  for complex polynomial
\begin{equation}\label{Q}
Q(z)\underset{\text{\tiny $z\in {\mathbb{C}}$}}{\equiv}\sum_{k=1}^{\lfloor p\rfloor}q_k z^k, \quad Q(0)=0, \quad \deg Q\leq \lfloor p\rfloor,
\end{equation}
we get  ${\rm Re}\,  Q(z)\overset{\eqref{tQ}}{=} O\bigl(|z|^{\lfloor p\rfloor+1}\bigr)$ as $z\to 0$. Hence, by Proposition \ref{prBC}, 
we have
\begin{equation*}
{\sf M}_{|Q|}(r)\leq 2{\sf M}_{{\rm Re}\,Q}(2r)+3\bigl|Q(0)\bigr|\overset{\eqref{Q}}{=}2{\sf M}_{{\rm Re}\,Q}(2r)\overset{\eqref{Q}}{=}O\bigl(|z|^{\lfloor p\rfloor+1}\bigr)\quad\text{as $z\to 0$.}
\end{equation*}
Thus, this  polynomial $Q$ of degree $\deg Q\overset{\eqref{Q}}{\leq} \lfloor p\rfloor$ has a root of multiplicity at least $\lfloor p\rfloor+1$ at $0$. Therefore, $Q=0$, and $q_k=0$ for each $k=1,\dots \lfloor p\rfloor$.  By definition 
\eqref{{ptp}Q} of $q_k$ and by Proposition  \ref{pr:bal}[b\ref{b3}],  the measure $\omega$ is a ${\sf mon}_p$-balayage of $\delta$. Evidently,  the inequality ${\sf pt}_{\omega}=p\geq {\sf pt}_{\delta}$ on ${\mathbb{C}}$ means that $\omega$ is also a $\ln_{{\mathbb{C}}}$-balayage of $\delta$. Thus, the second final part of our Theorem is proved under assumption \eqref{pd}.

If   $d\overset{\eqref{{pmu0+}p}}{\in} {\sf sbh} \cap  {\sf har} ({\mathbb{C}}\!\setminus\! S)$  is 
an arbitrary function with the Riesz measure $\delta$ from \eqref{d}, then, by the Weierstrass\,--\,Hadamard representation theorem,  $d$  admits representation
$d\overset{\eqref{d}}{=}{\sf pt}_{\delta}+h_d$, where $h_d$ is harmonic on ${\mathbb{C}}$. Consider function $\tilde d:={\sf pt}_{\delta}$ instead of function $d$, and function $\tilde v:=v-h_d$ instead of function $v$.  
For this pair of functions, we have $\delta =\varDelta_{\tilde d}$ and  $\omega=\varDelta_{\tilde v}$, 
and also 
\begin{equation*}
\tilde v(w)=v(w)-h_d(w) \underset{\text{\tiny $w\to \infty$}}{\overset{\eqref{{pmu0+}o}}{=}}
	 d(w)-h_d(w)+O\Bigl(\frac{1}{|w|^{\lfloor p\rfloor+1}}\Bigr)
= \tilde d(w)+O\Bigl(\frac{1}{|w|^{\lfloor p\rfloor+1}}\Bigr) \quad\text{as $w\to\infty$,}  
\end{equation*}
 but assumption  \eqref{pd} for the function $\tilde d$ instead of $d$ is already fulfilled.
\end{proof}

\section{Balayage with respect to  subharmonic functions of finite order}

\begin{definition}\label{defot}
	For  a function $f\colon [r,+\infty)\to {\overline {\mathbb{R}}}$ and $f^+:=\sup \{0,f\}$,  values 
{\rm (see \cite[\S~2]{KhaSch19})}
	\begin{subequations}\label{order+}
		\begin{align}
		{\sf ord}[f]&:= \limsup_{x\to +\infty} \frac{\ln (1+f^+(x))}{\ln x}\in {\mathbb{R}}^+\cup {+\infty},
		\tag{\ref{order+}o}\label{k{order+}o}
		\\
		{\sf type}_{p}[f]&:= 
		\limsup_{x\to +\infty} \frac{f(x)}{x^{p}}\in {\mathbb{R}}\cup {+\infty}
		\tag{\ref{order+}t}\label{k{order+}t}
		\end{align}
	\end{subequations}  
	are the \textit{order\/} of this function  $f$, and the \textit{type\/} of this function $f$  {\it under order\/} $p$, respectively.
\end{definition}

Using Definition \ref{defot}, for  $\nu \in {\sf Meas}$, we define 
\begin{subequations}\label{nuot}
	\begin{align}
	{\sf ord}[\nu]&\overset{\eqref{k{order+}o}}{:=}{\sf ord}\bigl[|\nu|^{{\text{\tiny\rm rad}}}\bigr], 
	\tag{\ref{nuot}o}\label{{nuot}o}
	\\
	{\sf type}_{p}[\nu]&\overset{\eqref{k{order+}t}}{:=}{\sf type}_{p}\bigl[|\nu|^{{\text{\tiny\rm rad}}}\bigr] \quad \text{if }
	{\sf ord}[\nu]\leq p \in {\mathbb{R}}^+.
	\tag{\ref{nuot}t}\label{{nuot}t}
	\end{align}
\end{subequations}

\begin{propos}\label{WA} 
	Let  $p\in {\mathbb{R}}^+$, and let    $\nu \in {\sf Meas}^+$ be a measure  with 
	\begin{equation}\label{tp}
	{\sf type}_p[\nu]\overset{\eqref{{nuot}t}}{<}+\infty
	\quad\text{under order  $p\overset{\eqref{{nuot}o}}{:=}{\sf ord}[\nu]\in {\mathbb{R}}^+$,}
\end{equation} 
	and  a finite measure $\omega \in {\sf Meas}^+$ satisfies the condition
\begin{equation}\label{imonl}
\left[
\begin{array}{cc}
{\sf N}^{\star}_{\nu}(1,+\infty; t\mapsto t^p)\overset{\eqref{{mB}star}}{<}+\infty &\text{when  $p\in {\mathbb{R}}^+\!\setminus\!{\mathbb{N}}_0$}, \\
{\sf N}^{\star}_{\nu}(1,+\infty; t\mapsto t^p\ln t)\overset{\eqref{{mB}star}}{<}+\infty &\text{when $p\in {\mathbb{N}}_0$}.
\end{array}
\right.
\end{equation}
Then there exist the following two equal repeated integrals
\begin{equation}\label{I}
-\infty \leq	\int \left(\int k_{\lfloor p \rfloor}(w,z) {\,{\rm d}} \nu (w) \right){\,{\rm d}} \omega (z)
	=\int \left(\int k_{\lfloor p \rfloor}(w,z) {\,{\rm d}} \omega (z) \right){\,{\rm d}} \nu (w)  <+\infty.
	\end{equation}
\end{propos}
\begin{remark} Condition \eqref{imonl} can be written in the equivalent form as
		\begin{equation}\label{imonl1}
	\int_1^{+\infty} \frac{ \omega^{{\text{\tiny\rm rad}}} (+\infty)-\omega^{{\text{\tiny\rm rad}}} (t)}{t}t^p 
\ln^{1+\lfloor p\rfloor -\lceil p\rceil} t {\,{\rm d}} t<+\infty, 
	\end{equation}
since, obviously, 
\begin{equation}\label{pp}
 1+\lfloor p\rfloor -\lceil p\rceil =
\begin{cases}
1 &\text{if } p\in {\mathbb{Z}}, \\ 
 0 &\text{if } p\in {\mathbb{R}}\!\setminus\! {\mathbb{Z}}.
\end{cases}
\end{equation} 
\end{remark}
\begin{proof} Standard classical estimates of  the  Weierstrass\,--\,Hadamard kernel 
	$k_{\lfloor p \rfloor}$ of genus $q$ from \eqref{kern:WA} give \cite[4.1.1]{HK}
	\begin{equation}\label{kq}
	k_{\lfloor p \rfloor}(w,z)\leq K_{\lfloor p \rfloor}(w,z)
	:={\rm const}_p^+ 
\begin{cases}
	\ln \bigl(1+|z|\bigr) &\text{ if $|w|<1$},\\
	\ln \bigl(1+|z|/|w|\bigr)&\text{ if $|w|\geq 1$ and  $\lfloor p\rfloor=0$},\\
	\bigl(|z|/|w|\bigr)^{\lfloor p\rfloor} \min \bigl\{ 1, |z|/|w|\bigr\}
	&\text{ if $|w|\geq 1$ and  $\lfloor p\rfloor\in{\mathbb{N}}$},
	\end{cases}
	\end{equation}
		The function $K_{\lfloor p \rfloor}$ is positive and Borel-measurable on ${\mathbb{C}}^2
	$. Using these estimates and the condition \eqref{tp}  we estimate, as in 
	\cite[Lemma 4.4]{HK},  the integral
	\begin{equation}\label{I1}
	\int K_{\lfloor p \rfloor}(w,z) {\,{\rm d}} \nu (w)
	\leq {\rm const}_{\nu, p}^+\cdot 
	\begin{cases}
	1+|z|^p\quad &\text{if  $p\notin {\mathbb{N}}_0$,}\\
	(1+|z|^p)\log (2+|z|)\quad &\text{if  $p\in {\mathbb{N}}_0$.}\\
	\end{cases}
	\end{equation}
	Hence, in view of \eqref{imonl}$\Leftrightarrow$\eqref{imonl1}, there exists the following repeated  integral 
	\begin{equation*}
	\int \left(\int K_{\lfloor p \rfloor}(w,z) {\,{\rm d}} \nu (w) \right){\,{\rm d}} \omega (z)<+\infty.
	\end{equation*}
	By the Fubini\,--\,Tonelli Theorem \cite[Ch.~V, \S~8, 1, Scholium]{Bourbaki} there are  two equal repeated integrals  
	\begin{equation*}
	\int \left(\int K_{\lfloor p \rfloor}(w,z) {\,{\rm d}} \nu (w) \right){\,{\rm d}} \omega (z)=
	\int \left(\int K_{\lfloor p \rfloor}(w,z) {\,{\rm d}} \omega (z) \right){\,{\rm d}} \nu (w)\in {\mathbb{R}}. 
	\end{equation*}
	Hence, in view of $k_{\lfloor p \rfloor}\overset{\eqref{kq}}{\leq} K_{\lfloor p \rfloor}$ on ${\mathbb{C}}^2$, by Fubini's Theorem \cite[Theorem 3.5]{HK} the integrals in  \eqref{I} exist and coincide
	with possible value of $-\infty$. 
\end{proof}

For $r\in {\mathbb{R}}^+$ and  $u\colon \partial \overline D(z, r)\to {\overline {\mathbb{R}}}$ we define
the integral average value of $u$ on the circle $\partial \overline D(z, r)$
\begin{subequations}\label{df:MCB}
	\begin{align} 
	{\sf C}_u(z,r)&:=\frac{1}{2\pi} \int_{0}^{2\pi}  u(z+re^{is}) {\,{\rm d}} s, \quad  {\sf C}_u^{{\text{\tiny\rm rad}}}(r):= {\sf C}_u(0,r),
	\tag{\ref{df:MCB}C}\label{df:MCBC}\\
	\intertext{and  the integral average value of   $u\colon \overline D(z,r)\to {\overline {\mathbb{R}}}$ on the disk  $\overline D(z, r)$}
	{\sf B}_u(z,r)&:=	\frac{2}{r^2}\int_0^r {\sf C}_u(z,t) t{\,{\rm d}} t
	, \quad {\sf B}_u^{{\text{\tiny\rm rad}}}(r):={\sf B}_u(0,r).
	\tag{\ref{df:MCB}B}\label{df:MCBB}
	\end{align}
\end{subequations} 

\begin{propos}\label{pr:rep}
Let $p\in {\mathbb{R}}^+$. If 
\begin{equation}\label{up}
u\in {\sf sbh}_*,  \quad {\sf type}_{\lfloor p\rfloor+1}[u]\overset{\eqref{{fc}c}}{=}0, \quad {\sf type}_{p}[{\sf C}_u^{{\text{\tiny\rm rad}}}]<+\infty,
\end{equation}
then ${\sf type}_p[\varDelta_u]<+\infty$,   there is a polynomial $P\in {\sf Pol}_p$  such that
\begin{equation}\label{rWA}
u(z)=\underset{I_u(z)}{\underbrace{\int_{{\mathbb{C}}} k_{\lfloor p\rfloor} (w,z){\,{\rm d}} \varDelta_u(w)}}+{\rm Re} \,P(z) \quad \text{for each $z\in {\mathbb{C}}$}, 
\end{equation} 
and 
\begin{equation}\label{uO}
u(z)\overset{\eqref{pp}}{=}O\Bigl(|z|^p\ln^{1+\lfloor p\rfloor-\lceil p\rceil} |z|\Bigr)\quad \text{as $z\to \infty$.}
\end{equation}
\end{propos}
\begin{proof}
By the Jensen\,--\,Privalov  formula \cite[Theorem 2.6.5.1]{Azarin} for subharmonic function $u$ on $\overline D(R)$,
${\sf N}^{{\text{\tiny\rm rad}}}_{\varDelta_u}(1,R)={\sf C}_u^{{\text{\tiny\rm rad}}}(R)-{\sf C}_u^{{\text{\tiny\rm rad}}}(1)$ for each $R>1$.
Hence, in view of  ${\sf type}_{p}[{\sf C}_u^{{\text{\tiny\rm rad}}}]\overset{\eqref{up}}{<}+\infty$,  we obtain
 ${\sf type}_p\bigl[{\sf N}^{{\text{\tiny\rm rad}}}_{\varDelta_u}(1,\cdot)\bigr]<+\infty$,  and, as an easy consequence, 
${\sf type}_p[\varDelta_u]<+\infty$. 
This gives the Weierstrass\,--\,Hadamard representation \cite{Arsove53p}, \cite{HK}, \cite{Azarin}
of the form \eqref{rWA}, but so far only with an entire function $P\neq 0$. 
The harmonic  function  ${\rm Re}\, P\overset{\eqref{rWA}}{:=} u-I_u$ is the difference between the  subharmonic function $u$ with ${\sf type}_{\lfloor p\rfloor+1}[u]=0$ and the canonical   Weierstrass\,--\,Hadamard  integral $I_u$ satisfying 
\begin{equation}\label{IuO}
I_u(z)\overset{\eqref{pp}}{=}O\Bigl(|z|^p\ln^{1+\lfloor p\rfloor-\lceil p\rceil} |z|\Bigr)\quad \text{as $z\to \infty$} 
\end{equation}
since ${\sf type}_p[\varDelta_u]<+\infty$ (see \eqref{kq}-\eqref{I1}). In particular, ${\sf type}_{\lfloor p\rfloor+1}[I_u]=0$.
Hence ${\sf type}_{\lfloor p\rfloor+1}[{\rm Re}\, P]={\sf type}_{\lfloor p\rfloor+1}[u-I_u] =0$ \cite[Theorems 2.9.3.2, 2.9.4.2]{Azarin}. Therefore ${\sf type}_{\lfloor p\rfloor+1}\bigl[{\sf M}_{|P|}\bigr]=0$ by Proposition \ref{prBC}. This is possible only if the entire function $P$
 is a polynomial of degree  $\deg P<\lfloor p\rfloor+1$, i.e.,  $P\in {\sf Pol}_p$, and 
 \eqref{uO} follows from \eqref{rWA} and \eqref{IuO}.
\end{proof}

\begin{theorem}\label{th:bu} Let $p\in {\mathbb{R}}^+$. If both  $\omega \in {\sf Meas}^+$ and $\delta\in {\sf Meas}^+$ satisfy \eqref{imonl}$\Leftrightarrow$\eqref{imonl1} for  $p\in {\mathbb{R}}^+$, and 
$\omega$ is a ${\sf mon}_p$-balayage of $\delta$, then for each function $u$  of \eqref{up}
 the integrals $\int_{\text{\tiny ${\mathbb{C}}$}} u{\,{\rm d}} \omega\in {\mathbb{R}}\cup -\infty$  
and $\int_{\text{\tiny ${\mathbb{C}}$}} u{\,{\rm d}} \delta \in {\mathbb{R}}\cup -\infty$
is well defined. If  $\omega$ is a ${\sf lnmon}_p$-balayage of $\delta$, then   
\begin{equation}\label{udo}
\int_{{\mathbb{C}}} u{\,{\rm d}} \delta\leq  \int_{{\mathbb{C}}} u{\,{\rm d}} \omega
\end{equation} 
for each $u$ of\/ \eqref{up},
i.e., $\omega$ is a balayage of $\delta$ with respect to the class of functions $u$ satisfying\/ \eqref{up}.
\end{theorem}
\begin{proof} 
 In view of\/  \eqref{imonl}, it follows from  Propositions\/ \ref{WA} and \ref{pr:rep} that  the integration of two summands in the right-hand side  of \eqref{rWA}  with respect to the measure $\omega$ is correct,  and by representation \eqref{rWA}. 
If  $\omega$ is a ${\sf lnmon}_p$-balayage of $\delta$, then 
\begin{multline*}
\int_{{\mathbb{C}}} u{\,{\rm d}} \delta\overset{\eqref{rWA}}{=}\int_{{\mathbb{C}}}
\int_{{\mathbb{C}}} k_{\lfloor p\rfloor} (w,z){\,{\rm d}} \varDelta_u(w){\,{\rm d}} \delta(z)+{\rm Re}  \int_{{\mathbb{C}}} P(z){\,{\rm d}} \delta(z)\\
=\Bigl|\text{Propositions \ref{pr:rep}, \ref{WA} with \eqref{imonl}, \ref{pr:bal}[b\ref{b3}]}\Bigr|
= \int_{{\mathbb{C}}}
\int_{{\mathbb{C}}} k_{\lfloor p\rfloor} (w,z){\,{\rm d}} \delta (z){\,{\rm d}} \varDelta_u(w)+{\rm Re}  \int_{{\mathbb{C}}} P(z){\,{\rm d}} \omega(z)
\\ \overset{\eqref{kpmtl}}{\leq}
\int_{{\mathbb{C}}} k_{\lfloor p\rfloor} (w,z){\,{\rm d}} \omega (z){\,{\rm d}} \varDelta_u(w)+{\rm Re}  \int_{{\mathbb{C}}} P(z){\,{\rm d}} \omega(z),
\end{multline*}
where the right-hand side is equal to  $\int_{\text{\tiny ${\mathbb{C}}$}} u{\,{\rm d}} \omega$ by 
Propositions \ref{WA} with \eqref{imonl}, and \ref{pr:rep}.
\end{proof}

\begin{remark}
The condition  ${\sf type}_{p}[{\sf C}_u^{{\text{\tiny\rm rad}}}]<+\infty$ from \eqref{up} in both Proposition \ref{pr:rep} and Theorem  \ref{th:bu} can be replaced by the condition  ${\sf type}_{p}[{\sf B}_u^{{\text{\tiny\rm rad}}}]\overset{\eqref{df:MCBB}}{<}+\infty$.
\end{remark}

\begin{acknowledgments}
The research  is funded in the framework of executing the development program of Scientific Educational Mathematical Center of 
Volga 
 Federal District by additional agreement no. 075-02-2020-1421/1 to agreement no. 075-02-2020-1421 (first author), and also was supported by a Grant of the Russian Foundation of Basic Research (Project no.~19-31-90007, second author).
\end{acknowledgments}

{\large \bf Compliance with ethical standards}

\paragraph*{\bf Conflict of interest.}
The authors declare that they have no conflict of interest.


\end{document}